\documentclass[onefignum,onetabnum]{siamart190516}

\usepackage{lipsum}
\usepackage{amsfonts}
\usepackage{graphicx}
\usepackage{epstopdf}
\usepackage{algorithm}
\usepackage{algorithmic}
\usepackage{subfigure}

\usepackage{xcolor}
\usepackage{rotating}
\usepackage{tikz}
\usepackage{tikz-qtree}
\usetikzlibrary{calc,er,positioning}
\usetikzlibrary{patterns,decorations.pathmorphing,decorations.markings}
\ifpdf
  \DeclareGraphicsExtensions{.eps,.pdf,.png,.jpg}
\else
  \DeclareGraphicsExtensions{.eps}
\fi

\newcommand{\code}[1]{\texttt{#1}}

\newsiamremark{remark}{Remark}
\newsiamremark{example}{Example}
\newsiamremark{hypothesis}{Hypothesis}
\crefname{hypothesis}{Hypothesis}{Hypotheses}
\newsiamthm{claim}{Claim}
\newsiamremark{definitionn}{Definition}

\headers{D. Collin}{D. Collin}

\title{How to find all connections in the Pantelides algorithm for delay differential-algebraic equations\thanks{Submitted February 3, 2022.}}

\author{Daniel Collin\thanks{Technische Universität Berlin, Berlin, Germany 
  (\email{daniel.collin@web.de}).}}

\usepackage{amsopn}

\definecolor{verylightgray}{gray}{0.30}
\tikzstyle{vertex}=[circle, draw, inner sep=0pt, minimum size=6pt]
\newcommand{\vertex}{\node[vertex]}

\newcommand{\enode}[3][]{
	\ifthenelse{\equal{\detokenize{#1}}{\detokenize{try}}}
	{\def\opt{fill}}
	{	
		\def\opt{}	
	}
	\vertex[\opt] (de#2) at (0,10-#2) [label=left:#3] {};
}
\newcommand{\vnode}[4][black]{
	\ifthenelse{\equal{\detokenize{#1}}{\detokenize{low}}}
	{\def\difcolor{verylightgray}}
	{	
		\def\difcolor{black}	
	}
	
	\vertex (dv#2) at (2,10-#2) [label=right:\textcolor{\difcolor}{#3}] {};
	\node (n#2) at (3,10-#2) [label=right:#4] {};
}

\newcommand{\edge}[3][black]{
	\ifthenelse{\equal{\detokenize{#1}}{\detokenize{assign}}}
	{ \def\difColor{blue} \def\type{solid} \def\thickness{very thick}}
	{
		\ifthenelse{\equal{\detokenize{#1}}{\detokenize{low}}}
		{ \def\difColor{verylightgray} \def\type{dashed} \def\thickness{very thin}}
		{
			\ifthenelse{\equal{\detokenize{#1}}{\detokenize{con}}}
	        { \def\difColor{red} \def\type{solid} \def\thickness{very thick}}
	        {
			    \def\difColor{black} \def\type{solid}	\def\thickness{thin}
		    }
		}
	}			
	\path
	(d#2) edge[color = \difColor, \type, \thickness] (d#3);
}

\begin{document}

\maketitle

\begin{abstract}
  The Pantelides algorithm for delay differen\-tial-algebraic equations (DDAEs) is a method to structurally analyse such systems with the goal to detect which equations have to be differentiated or shifted to construct a solution. In this process, one has to detect implicit connections between equations in the shifting graph, making it necessary to check all possible connections. The problem of finding these efficiently remained unsolved so far. It is explored in further detail and a reformulation is introduced. Additionally, an algorithmic approach for its solution is presented.
\end{abstract}

\begin{keywords}
  delay differential-algebraic equation, Pantelides algorithm, structural analysis, enumeration algorithm, spanning tree
\end{keywords}

\begin{AMS}
  05C30, 34A09, 34K32, 65L80
\end{AMS}

\section{Introduction}

Delay differential-algebraic equations (DDAEs) are a class of differential equations for some function $x(t)$ on a time interval $[0,T)\subseteq\mathbb{R}$, $T>0$, that, in their simplest form, not only depend on the time derivative $\dot x(t)$ but also on a previous time state $\Delta_{-\tau} x(t):=x(t-\tau)$ with a delay $\tau>0$. Additionally, the system may possess algebraic constraints such that it can only be formulated in implicit form. Here, the DDAE is assumed to be a system of $n\in\mathbb{N}$ equations~${F=(F_1,...,F_n)}$ and variables~${x=(x_1,...,x_n)}$. Thus, consider a DDAE of the form
\begin{equation}
    \label{eq:ddae}
    F(t,x(t),\dot x(t), \Delta_{-\tau} x(t))=0,
\end{equation}
where 
\begin{equation*}
    x:[-\tau,T) \to \mathbb{R}^n \quad \text{and} \quad F:[0,T)\times \mathbb{D}_{x}\times \mathbb{D}_{\dot x}\times \mathbb{D}_{\Delta x} \to \mathbb{R}^n,
\end{equation*}
with $\mathbb{D}_{x},\mathbb{D}_{\dot x},\mathbb{D}_{\Delta x}\subseteq \mathbb{R}^n$ being open. Here, $\dot{x}$ denotes the derivative of $x$ with respect to $t$ from the right.
To obtain an initial value problem, \cref{eq:ddae} has to be equipped with an initial condition
\begin{equation}
    \label{eq:init}
    x(t)=\phi(t) \quad \text{for} \quad t\in[-\tau,0].
\end{equation}

Equations of this form arise in many applications, such as multibody control systems, electric circuits or fluid dynamics (see \cite{ph,unger20b}). They combine features of delay differential equations (DDEs) and differential-algebraic equations (DAEs), which makes them particularly difficult to solve. 
Solutions may depend on derivatives of $F$ and on evaluations of $F$ at future time points (see \cite{campbell95,phi12}). Therefore, the interplay of the differentiation operator~$\tfrac{\mathrm{d}}{\mathrm{d}t}$ and the shift operator $\Delta_{-\tau}$ has to be treated carefully (cf. \cite{phi14}) and solutions have to be constructed by differentiating and shifting equations (cf. \cite{campbell95,phi12,phi14,phi16,unger18,trenn19}). Even for linear DDAEs, general existence and uniqueness results can only be obtained using a distributional solution concept (see \cite{trenn19,unger20b}) or imposing further restrictions on the DDAE (see \cite{phi12,phi16,unger18}) or its initial function in \cref{eq:init} (see \cite{phi18,unger18}). For nonlinear DDAEs like \cref{eq:ddae}, solutions can be established for certain classes (cf. \cite{ascher95,unger20}).

In most cases, the construction of a solution for a DDAE involves the method of steps. This means that the equation is successively integrated over the time intervals~${\left[i\tau,(i+1)\tau\right)}$, $i\in\mathbb{N}$. By substituting the delayed variables with the already computed solution of the previous interval, the problem can be reduced to solve a DAE in each step (cf. \cite{campbell80,bellen03,phi16}). However, this method does not always succeed and a reformulation of \cref{eq:ddae} is required such that the DAE that has to be solved in each interval is regular and has a small index. Hereby, the index is, roughly speaking, a measure how often parts of the DAE have to be differentiated to reformulate the DAE as an ordinary differential equation. This reformulation can be done by a compress-and-shift algorithm (see \cite{campbell95,trenn19}) or by a combined shift and derivative array (see \cite{phi16}). The differentiation and  shifting of certain equations is necessary in both cases.

Determining which equations one has to differentiate or shift is therefore a central aspect of the solution process of a DDAE. In \cite{ahrens20}, the Pantelides algorithm for delay differential-algebraic equations is presented as a tool to exploit the structure of (\ref{eq:ddae}), i.e., the information which variable appears in which equation, to determine the number of differentiations and shifts necessary to solve the DDAE. 
It is based on the Pantelides algorithm for DAEs (see \cite{pantelides88}) and will be simply referred to as Pantelides algorithm from now on. The approach consists of defining different bipartite graphs, where each equation and certain equivalence classes of variables are represented by the nodes. Edges exist between equation nodes and variable nodes if and only if one of the variables of the equivalence class appears in that equation. In other words, the graphs represent the structure of the DDAE. Then, matchings between equation nodes and variable nodes of highest shift and differentiation order are constructed in these graphs. This is achieved by following a specific pattern of shifting and differentiating equations and the variables belonging to it. At the end of the process, each equation can be resolved for a variable of highest shift and differentiation order. For more details on the whole procedure and the algorithm, see the original paper \cite{ahrens20}.

In this work, the focus is put on a specific subproblem that appears in the Pantelides algorithm. During the first part of the algorithm, called the shifting step, one has to shift equations that are connected to each other in a certain way through edges in a specific graph, called the shifting graph. Generally, such a connection is not unique and all possible connections have to be found to allow for correct shifting. The identification of all of these connections, however, may be computationally very expensive, and no efficient algorithmic solution is known so far. 

The contribution of this paper is a deeper exploration of this problem. First, it is explained in further detail and all important preliminaries are given in \cref{sec:2}. Then, a new solution approach is proposed based on the reformulation to a known enumeration problem from graph theory. The equivalence of both problems is proven (\cref{sec:3}). Additionally, an algorithm from \cite{gabow78} for the solution of the enumeration problem is presented (\cref{sec:4}). This algorithm is applied to the original problem, yielding a method for finding all connections in the Pantelides algorithm, and two detailed examples of its usage are given (\cref{sec:5}). Finally, a numerical demonstration of the advantageous properties of the new method is shown (\cref{sec:6}) and the paper is concluded with a summary and some final remarks (\cref{sec:7}).
\bigbreak
\textbf{Notation:}
The natural numbers, the non-negative integers, and the reals are denoted by $\mathbb{N}$, $\mathbb{N}_0$ and $\mathbb{R}$, respectively. For a differentiable function $f: \mathbb{I}\to\mathbb{R}^n$, the notation~${\dot f:=\tfrac{\mathrm{d}}{\mathrm{d}t}f}$ is used to denote the derivative with respect to the (time) variable $t$ and $\ddot f:=\frac{d}{dt}\dot f$ for the second derivative. For higher derivatives of order $q\in\mathbb{N}_0$, the abbreviation $f^{(q)}$ is used. Similarly, the shift operator~$\Delta_\tau$ is defined as ${\Delta_\tau f(t) = f(t+\tau)}$. The union of two sets $A$ and $B$ is denoted by $A \cup B$. A disjoint union of sets is written as~${A\dot\cup B}$. The cardinality of the set $A$ is denoted by $|A|$.

\section{Problem description}
\label{sec:2}

This section describes the overall problem of the paper in detail and introduces the most important definitions to give all preliminaries needed to understand the solution approach. Since this paper can be seen as an extension of \cite{ahrens20}, most information is based on that work and all derivations can be found there.

The Pantelides algorithm translates the structural information of the DDAE into graphs. First, the \textit{shifting graph} $G^S$ is constructed by combining all variables of the same index $k$ (for each $k=1,...,n$) and shift order~${p\in\mathbb{N}_0\cup\{-1\}}$ (but possibly different differentiation order) into the same equivalence class, i.e.,
\begin{equation*}
    [\Delta_{p\tau}x_k]:=\left\{\Delta_{p\tau}x_k,\Delta_{p\tau}\dot x_k,\Delta_{p\tau} \ddot x_k,...\right\}.
\end{equation*}
Then, one can define the set of equation nodes, variable nodes, and edges as
\begin{align*}
    V^S_E&:=\left\{F_1,...,F_n\right\},\\
    V^S_V&:=\Big\{[\Delta_{p\tau}x_k]\;\Big|\; \exists k,p\in\mathbb{N}_0\cup\{-1\} \text{ s.t. }\Delta_{p\tau}x_k^{(q)}\\
     &\qquad\qquad\qquad\;\;\text{ appears in DDAE for any }q\in\mathbb{N}_0\Big\},\\
    E^S&:=\left\{\{F_i,v_k\}\in V^S_E\times V^S_V \;\Big| \; \exists \tilde x \in v_k\text{ that appears in }F_i \right\},
\end{align*}
respectively, which yields the shifting graph defined as~$G^S:=(V^S_E \dot \cup V^S_V,E^S)$.

In the bipartite shifting graph, one successively assigns to each equation node $F_i\in V^S_E$ an equivalence class $v_k\in V^S_V$ of highest shift, i.e., if $v_k=[\Delta_{p\tau}x_k]$ is of highest shift and occurs in $F_i$, then $[\Delta_{(p+\ell)\tau}x_k]$, for $\ell>0$, does not occur in any equation. By definition, a variable node $[\Delta_{p\tau}x_k]$ with negative shift $p=-1$ is never of highest shift and cannot be matched to an equation node.  Like that, a matching $\mathcal{M}$ is constructed, consisting of all assigned pairs $\{F_i,v_k\}$. If a particular $F_j$ cannot be matched to a variable node that is not in $\mathcal{M}$ yet, the node $F_j$ is called \textit{exposed with respect to $\mathcal{M}$}. The corresponding equation is shifted, together with all other equations that $F_j$ is connected to via alternating paths with respect to $\mathcal{M}$ in~$G^S$. An \textit{alternating path with respect to $\mathcal{M}$} is a sequence of edges
\begin{equation*}
    \left(\left\{F_{i_1},v_{k_1}\right\},\left\{v_{k_1},F_{i_2}\right\},\left\{F_{i_2},v_{k_2}\right\},...,\left\{v_{k_{N-1}},F_{i_N}\right\}\right)
\end{equation*}
in $G^S$, where all $i_\ell$, $\ell=1,...,N$, and all $k_m$, $m=1,...,N-1$, are distinct, respectively, and that has alternating non-matching and matching edges while starting with a non-matching edge.

However, simply shifting all these equations may not be sufficient, since the connection may be given only implicitly through the equivalence classes of the variable nodes. To see this, define $G:=(V_E\dot\cup V_V,E)$ as the \textit{graph of the DDAE} with
\begin{align*}
    V_E&:=\left\{ F_1,...,F_n\right\},\\
    V_V&:=\Big\{\Delta_{p\tau}x_k^{(q)}\;\Big|\; \exists k,p,q\in\mathbb{N}_0\cup\{-1\} \text{ s.t. }\\
     &\qquad\qquad\qquad\;\;\Delta_{p\tau}x_k^{(q)}\text{ appears in DDAE}\Big\},\\
    E&:=\left\{\left.\{F_i,v_k\}\in V_E\times V_V \;\right| \; v_k\text{ appears in }F_i \right\}.
\end{align*}
In other words, the graph of the DDAE contains all variables explicitly as distinct nodes without using equivalence classes. An implicit connection in the shifting graph means that the involved equations contain variables with the same shift but a different differentiation order. Thus, they belong to the same variable node in the shifting graph but not to the same node in the graph of the DDAE. There is an alternating path connecting the exposed equation $F_j$ and the equation that has to be shifted in $G^S$ but not in $G$. In this case, an explicit connection has to be established by differentiating the involved equations that do not depend on the highest derivative in the equivalence class. To ensure that all implicit connections are resolved, all possible connections have to be identified and checked.

Theoretically, this could be done by just checking all possible combinations of edges of $E^S$ that yield alternating paths. In practice, however, this approach is not feasible, because the amount of combinations increases rapidly with the number of nodes and edges of the graph~$G^S$. By reformulating the problem, it can be solved much more efficiently.

\begin{figure}[htbp]
    \centering
    \subfigure[The shifting graph of \cref{eq:example1}.]{
        \centering
        \begin{tikzpicture} 
        	\enode[try]{1}{$F_1$}
    		\enode[try]{2}{$F_2$}
    		\enode[try]{3}{$F_3$}
    		\vnode{1}{$x_1,\dot x_1$}{}
    		\vnode{2}{$x_2$}{}
    		\vnode[low]{3}{$\Delta_{-\tau} x_3$}{}
    		\edge[assign]{e1}{v1}
    		\edge[assign]{e2}{v2}
    		\edge{e2}{v1}
    		\edge{e3}{v1}
    		\edge{e3}{v2}
    		\edge[low]{e3}{v3}
        \end{tikzpicture}
    	\label{fig:example1.1}
    }
    \subfigure[The first connection.]{
        \centering
        \begin{tikzpicture} 
        	\enode[try]{1}{$F_1$}
    		\enode[try]{2}{$F_2$}
    		\enode[try]{3}{$F_3$}
    		\vnode{1}{$x_1,\dot x_1$}{}
    		\vnode{2}{$x_2$}{}
    		\vnode[low]{3}{$\Delta_{-\tau} x_3$}{}
    		\edge[con]{e1}{v1}
    		\edge[con]{e2}{v2}
    		\edge[con]{e2}{v1}
    		\edge{e3}{v1}
    		\edge[con]{e3}{v2}
    		\edge[low]{e3}{v3}
        \end{tikzpicture}
    	\label{fig:example1.2}
    }
    \subfigure[The second connection.]{
        \centering
        \begin{tikzpicture} 
        	\enode[try]{1}{$F_1$}
    		\enode[try]{2}{$F_2$}
    		\enode[try]{3}{$F_3$}
    		\vnode{1}{$x_1,\dot x_1$}{}
    		\vnode{2}{$x_2$}{}
    		\vnode[low]{3}{$\Delta_{-\tau} x_3$}{}
    		\edge[con]{e1}{v1}
    		\edge[con]{e2}{v2}
    		\edge{e2}{v1}
    		\edge[con]{e3}{v1}
    		\edge[con]{e3}{v2}
    		\edge[low]{e3}{v3}
        \end{tikzpicture}
    	\label{fig:example1.3}
    }
    \subfigure[The graph of the DDAE of \cref{eq:example1}.]{
        \centering
        \begin{tikzpicture} 
        	\enode[try]{1}{$F_1$}
    		\enode[try]{2}{$F_2$}
    		\enode[try]{3}{$F_3$}
    		\vnode{1}{$x_1$}{}
    		\vnode{2}{$x_2$}{}
    		\vnode{3}{$\dot x_1$}{}
    		\vnode{4}{$\Delta_{-\tau} x_3$}{}
    		\edge{e1}{v3}
    		\edge{e2}{v3}
    		\edge{e2}{v2}
    		\edge{e3}{v1}
    		\edge{e3}{v2}
    		\edge{e3}{v4}
        \end{tikzpicture}
    	\label{fig:example1.4}
    }
    \caption{Visualization of the problem of finding all connections for $F_3$ with respect to $\mathcal{M}$ in the shifting step of the DDAE \cref{eq:example1}.}
    \label{fig:example1}
\end{figure}
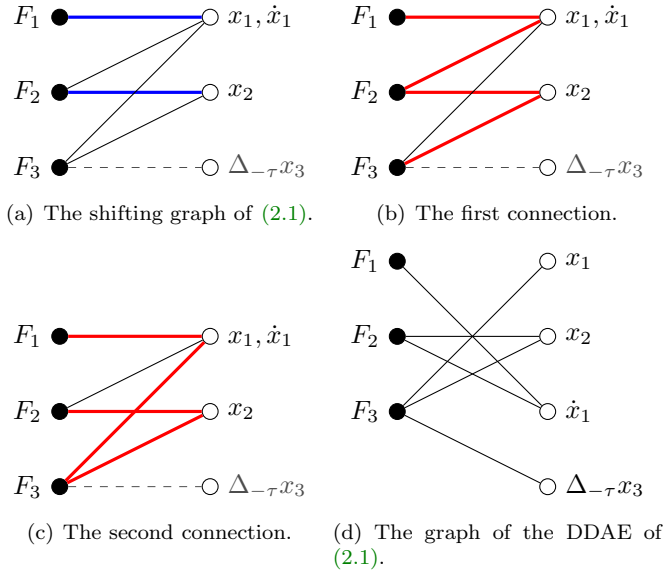

\begin{example}
\label{ex:con}
For an illustration of the problem, consider the DDAE from \cite[Example 3.11, p.17]{ahrens20}:
\begin{equation}\label{eq:example1}
    \begin{aligned}
    \dot x_1 &= f_1,\\
    \dot x_1 &= x_2+ f_2,\\
    0 &= x_1 +x_2+ \Delta_{-\tau} x_3 + f_3.
    \end{aligned}
\end{equation}
After assigning the equivalence classes $[x_1]=\{x_1,\dot x_1\}$ to~$F_1$ and~$[x_2]=\{x_2\}$ to~$F_2$ in the shifting step, this yields the matching
\begin{equation*}
    \mathcal{M}=\left\{\{F_1,\{x_1,\dot x_1\}\},\{F_2,\{x_2\}\}\right\}
\end{equation*}
and the shifting graph $G^S$ in figure \Cref{fig:example1.1} (matching edges are colored in blue). Equation $F_3$ is exposed and cannot be matched directly to any equivalence class, but it is connected via alternating paths to the other equation nodes.
Therefore, all possible connections for~$F_3$ with respect to $\mathcal{M}$ have to be found. These are
\begin{equation*}
    \mathcal{C}_1=\left\{(F_3,\{x_2\},F_2),(F_2,\{x_1,\dot x_1\},F_1)\right\}
\end{equation*}
(pictured red in \Cref{fig:example1.2}) and
\begin{equation*}
    \mathcal{C}_2=\left\{(F_3,\{x_1,\dot x_1\},F_1),(F_3,\{x_2\},F_2)\right\}
\end{equation*}
(pictured red in \Cref{fig:example1.3}). Additionally, $G$, the graph of the DDAE, is visualized in \Cref{fig:example1.4}. One can see that connection $\mathcal{C}_1$ does also exist in $G$ via the path
\begin{equation*}
    \left(\{F_3,x_2\},\{x_2,F_2\},\{F_2,\dot x_1\},\{\dot x_1,F_1\}\right)
\end{equation*}
and hence, is an explicit connection. However, $\mathcal{C}_2$ is implicit, as the alternating path~${(F_3,\{x_1,\dot x_1\},F_1)}$ does not connect $F_3$ and $F_1$ in $G$. One has to differentiate $F_3$ to establish an explicit connection. It can be seen that checking one connection is not enough, all have to be identified to resolve possible implicit connections in the shifting graph.
\end{example}

\section{Reformulation of problem}
\label{sec:3}

First, a connection has to be technically defined. To simplify the notation, let $G=(V_E \dot\cup V_V,E)$ be a bipartite graph with equation nodes $V_E$ and variable nodes $V_V$, where $E$ contains only edges between $V_E$ and $V_V$, not between nodes of one set. Further, let a matching 
\begin{equation*}
    \mathcal{M}=\left\{\left\{F_{i_1},v_{k_1}\right\},..., \left\{F_{i_M},v_{k_M}\right\}\right\}\in E^M
\end{equation*}
be given in $G$ with $M<n$ and $F_j\in V_E$ as an exposed equation node with respect to $\mathcal{M}$. Then,
\begin{align*}
    C_{F_j}:=\big\{ \left.F_k\in V_E \;\right|\; \exists\text{ alternating path }\text{between } F_j \text{ and } F_k \text{ in } G \big\}
\end{align*}
denotes the equation nodes that are connected to $F_j$ via an alternating path. This set is automatically generated by the algorithm Augmentpath (see \cite[Algorithm 1, p.10]{ahrens20}, \cite[Algorithm 3.2, p.217]{pantelides88}).
\begin{definitionn}
\label{def:con}
Let $G=(V_E \dot\cup V_V,E)$ be a bipartite graph and $\mathcal{M}$ a matching in $G$. Further, let $F_j\in V_E$ be exposed with respect to $\mathcal{M}$ and $C_{F_j}$ as defined above. A \textit{connection for $F_j$ with respect to $\mathcal{M}$} is defined as a set of connected alternating paths $(F_i,v_k,F_\ell) \in V_E \times V_V \times V_E$, with~$\{F_i,v_k\} \in  E \setminus \mathcal{M}$ and $\{v_k,F_\ell\}\in\mathcal{M}$. Additionally, it has to hold that for all~$F_\ell\in C_{F_j}$ the corresponding matching edge~$\{v_k,F_\ell\} \in \mathcal{M}$ occurs exactly once and there is at least one alternating path starting in $F_j$.
\end{definitionn}

Note that the definition of a connection for $F_j$ with respect to $\mathcal{M}$ has been changed in comparison to the definition from \cite[p.18]{ahrens20}. The previous definition allows sets of alternating paths that contain cycles and not necessarily the exposed node $F_j$. In the forthcoming \Cref{cor:cycle-free}, it is shown that a connection in the sense of \Cref{def:con}, however, is cycle-free. Also, a connection for $F_j$ with respect to $\mathcal{M}$ will simply be referred to as a connection when it is clear which node is exposed and which matching the connection is based on.

With the exact definition of a connection, one can further define the connection graph by interpreting the alternating paths from this definition as directed edges between the equation nodes.

\begin{definitionn}
\label{def:congraph}
Let $G=(V_E \dot\cup V_V,E)$ be a bipartite graph and $\mathcal{M}$ a matching in $G$. Further, let $F_j\in V_E$ be exposed with respect to $\mathcal{M}$ and $C_{F_j}$ as defined above. Define the set of nodes $V_H:=C_{F_j}\dot\cup \{F_j\}$ and directed edges
\begin{align*}
    E_H:=\big\{&\left. (F_i,F_\ell)\in V_H \times V_H \;\right| \; (F_i,v_k,F_\ell) \textnormal{ is an alternating}\\
    &\text{path with }\{F_i,v_k\}\in E\setminus \mathcal{M},\{v_k,F_\ell\}\in\mathcal{M}\big\}.
\end{align*}
Then, the directed graph $H:=(V_H,E_H)$ is called \textit{connection graph for $F_j$ with respect to $\mathcal{M}$}.
\end{definitionn}

\begin{remark}
\label{rem:var}
Denoting the alternating path $(F_i,v_k,F_\ell)$ as~$(F_i,F_\ell)$, it might seem as if information is lost about which variable node $v_k$ connects the equation nodes~$F_i$ and $F_\ell$. However, since each $F_\ell$ is uniquely matched to one~$v_k$ in $\mathcal{M}$, the variable node can easily be reconstructed from the directed edge $(F_i,F_\ell)$ using $\mathcal{M}$. A second approach to not lose information is to define edge weights~$w_{i\ell}=k$ for the edges $(F_i,F_\ell)$, i.e., if $v_k$ is to be reconstructed from $(F_i,F_\ell)$, it holds that $v_k=v_{w_{i\ell}}$.
\end{remark}

\noindent Similar to before, the connection graph for $F_j$ with respect to $\mathcal{M}$ will be referred to simply as connection graph when it is clear which node is exposed and which matching the connection graph is based on.

\begin{example}
\label{ex:congraph}
Consider again the DDAE \cref{eq:example1} from \Cref{ex:con} and the shifting graph from \Cref{fig:example1.1}. Based on the matching
\begin{equation*}
    \mathcal{M}=\left\{\{F_1,\{x_1,\dot x_1\}\},\{F_2,\{x_2\}\}\right\}
\end{equation*}
and $C_{F_3}=\{F_1,F_2\}$, the connection graph for $F_3$ with respect to $\mathcal{M}$ can be defined according to \Cref{def:congraph} as~$H=(V_H,E_H)$ with
\begin{align*}
    V_H&=\{F_1,F_2,F_3\},\\
    E_H&=\left\{(F_2,F_1),(F_3,F_1),(F_3,F_2)\right\}.
\end{align*}
A picture of $H$ can be seen in \Cref{fig:example2}.
\end{example}

\begin{figure}[b!]
    \centering
    \begin{tikzpicture}[node distance={20mm}, roundnode/.style={circle, draw=black, thin}]
	    \node[roundnode](3){$F_3$};
	    \node[roundnode](1)[right=of 3]{$F_1$};
	    \node[roundnode](2)[below=of 3]{$F_2$};
	    \draw[->,>=stealth,thick] (3) -- (1);
	    \draw[->,>=stealth,thick] (3) -- (2);
	    \draw[->,>=stealth,thick] (2) -- (1);
	\end{tikzpicture}
    \caption{The connection graph for $F_3$ with respect to~$\mathcal{M}$ of the DDAE \cref{eq:example1}.}
    \label{fig:example2}
\end{figure}
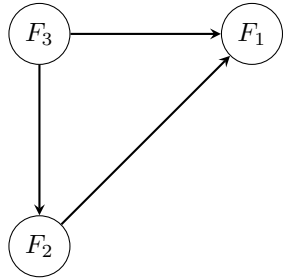

The connection graph facilitates to reformulate the problem of finding all connections in the sense of \Cref{def:con} by transferring it from the shifting graph to the connection graph. It translates to finding all arborescences (defined below) with root $F_j$ in $H$. To prove this, some definitions and lemmas from graph theory are needed (see \cite[p.71-73]{algmath} for reference and proofs).

\begin{definitionn}
\label{def:arb}
\cite[Definition 6.15, p.71, Definition 6.20, 6.22, p.73]{algmath}
\begin{enumerate}
    \item An undirected graph is called \textit{forest} if it contains no cycles.
    \item An undirected graph is called \textit{tree} if it is a forest and connected.
    \item Given a directed graph $G$, one can replace every directed edge by an undirected edge to get an undirected graph. The arising graph is called the \textit{underlying undirected graph} of $G$.
    \item A directed graph is called \textit{branching} if its underlying undirected graph is a forest and every node has at most one edge ending in it.
    \item A directed graph is called \textit{arborescence} if it is a connected branching. 
\end{enumerate}

\end{definitionn}

\begin{lemma} 
\label{lem:tree}
\textnormal{\cite[Theorem 6.18, p.72]{algmath}}
Let $G$ be an undirected graph with $n$ vertices. Then, $G$ is a tree if and only if $G$ has $n-1$ edges and is connected.
\end{lemma}

\noindent The underlying undirected graph of an arborescence has to be a connected forest, i.e., a tree. According to \Cref{lem:tree}, an arborescence with $n$ nodes thus has~$n-1$ edges and, according to \Cref{def:arb}, every node has at most one edge ending in it. Therefore, there is exactly one node $r$ with no incoming edge. Let $\delta^-(v)$ be the set of incoming edges of a vertex~$v$ of $G$. Then, this condition can be formulated as ${\delta^-(r)=\emptyset}$. In this case,~$r$ is called \textit{root} of the arborescence. For any edge $(u,v)$ of an arborescence, $v$ is called a \textit{child} of $u$ and $u$ the \textit{predecessor} of~$v$. Vertices with no children are called \textit{leaves} (see \cite[p.73]{algmath}).

\begin{lemma}
\label{lem:arb}
\textnormal{\cite[Theorem 6.23, p.73]{algmath}}
Let $G$ be a directed graph and $r$ a vertex of $G$. Then, the following statements are equivalent:
\begin{enumerate}
    \item $G$ is an arborescence with root $r$.
    \item $G$ is a branching and $\delta^-(r)=\emptyset$.
    \item $\delta^-(r)=\emptyset$ and there exists a uniquely determined directed path from $r$ to every vertex in $G$.
\end{enumerate}
\end{lemma}

\begin{theorem}
\label{thm:con}
Let $G=(V_E \dot\cup V_V,E)$ be a bipartite graph,~$\mathcal{M}$ a matching in $G$, $F_j\in V_E$ exposed with respect to $\mathcal{M}$, and~$H=(V_H,E_H)$ the connection graph for $F_j$ with respect to $\mathcal{M}$. Then, the following statements are equivalent:
\begin{enumerate}
    \item The set
    \begin{equation*}
        \mathcal{C}:=\left\{\left(F_{i_1},v_{k_1},F_{\ell_1}\right),...,\left(F_{i_M},v_{k_M},F_{\ell_M}\right)\right\}
    \end{equation*}
    is a connection for $F_j$ with respect to $\mathcal{M}$.
    \item The directed subgraph
    \begin{equation*}
        \mathcal{H}:= \left(V_H,\left\{\left(F_{i_1},F_{\ell_1}\right),...,\left(F_{i_M},F_{\ell_M}\right)\right\}\right)\subseteq H
    \end{equation*}
    is an arborescence with root $F_j$.
\end{enumerate}
\end{theorem}

\begin{proof} 
$\Rightarrow$ ii): Let $F_j\in V_E$ be exposed, denote a connection for~$F_j$ as ~$\mathcal{C}=\{(F_{i_1},v_{k_1},F_{\ell_1}),...,(F_{i_M},v_{k_M},F_{\ell_M})\}$, and let $H$ be the connection graph for $F_j$, with respect to $\mathcal{M}$, respectively. Additionally, denote
\begin{equation*}
    \mathcal{H}=(V_H,E_\mathcal{H}),\;\;\;\; E_\mathcal{H}:=\left\{\left(F_{i_1},F_{\ell_1}\right),...,\left(F_{i_M},F_{\ell_M}\right)\right\}\subseteq E_H,
\end{equation*}
as a subgraph of $H$, and
\begin{equation*}
    \mathcal{H}_u:=(V_H,E_u),\quad E_u:=\left\{\left\{F_{i_1},F_{\ell_1}\right\},...,\left\{F_{i_M},F_{\ell_M}\right\}\right\},
\end{equation*}
as the underlying undirected graph of $\mathcal{H}$. According to \Cref{def:congraph}, it holds that ${|V_H|=|C_{F_j}|+1}$ and therefore,
\begin{equation*}
    M=|\mathcal{C}|=|C_{F_j}|=|V_H|-1.
\end{equation*} 
It follows that $ \mathcal{H}_u$ has $|V_H|-1$ edges.
According to \Cref{def:con}, the alternating paths $(F_i,v_k,F_\ell)\in\mathcal{C}$ are connected and each~$F_\ell~\in~C_{F_j}$ occurs exactly once. The variable nodes~$v_k$ are uniquely determined by the equation nodes $F_\ell$ via the matching~$\mathcal{M}$. This implies that each~$v_{k_m}$, for $m=1,...,M$, also occurs only once in $\mathcal{C}$ and the alternating paths must be connected via the equation nodes $F_i$ and $F_\ell$. Thus, the edges in $E_u$ are connected and thereby, $\mathcal{H}_u$ is connected as well.

In summary, $\mathcal{H}_u$ has $|V_H|-1$ edges and is connected. Therefore, according to \Cref{lem:tree}, the underlying undirected graph of $\mathcal{H}$ is a tree (and also a forest). Additionally, $V_H=C_{F_j}\dot\cup \{F_j\}$ and it was already mentioned that each $F_\ell\in C_{F_j}$ occurs exactly once as second equation node in the alternating paths $(F_i,v_k,F_\ell)\in\mathcal{C}$. The node~$F_j$ itself has no alternating path leading to it because it is exposed. Therefore, every node of $\mathcal{H}$ has at most one edge ending in it. According to \Cref{def:arb}, the graph $\mathcal{H}$ is a branching.

Finally, one can choose the exposed node $F_j$ as the root of~$\mathcal{H}$ because $\delta^-(F_j)=\emptyset$. By \Cref{lem:arb}, it follows that~$\mathcal{H}$ is an arborescence with root $F_j$.

ii) $\Rightarrow$ i): Let $\mathcal{H}= (V_H,\{(F_{i_1},F_{\ell_1}),...,(F_{i_M},F_{\ell_M})\})\subseteq H$ be an arborescence with root $F_j$, and denote
\begin{equation*}
    \mathcal{C}=\left\{\left(F_{i_1},v_{k_1},F_{\ell_1}\right),...,\left(F_{i_M},v_{k_M},F_{\ell_M}\right)\right\}.
\end{equation*}
Since $\mathcal{H}$ is a subgraph of the connection graph for $F_j$ with respect to $\mathcal{M}$, it follows that all $(F_i,v_k,F_\ell)\in\mathcal{C}$ are alternating paths with $\{F_i,v_k\}\in E\setminus \mathcal{M}$, $\{v_k,F_\ell\}\in\mathcal{M}$, where $v_k$ is uniquely determined by the matching $\mathcal{M}$ (see \Cref{def:congraph} and \Cref{rem:var}). From \Cref{lem:arb}, it is known that the root $F_j$ of the arborescence is a vertex of~$\mathcal{H}$ with $\delta^-(F_j)=\emptyset$ and, therefore, is also included as the starting node in an alternating path from $\mathcal{C}$.

Additionally, there exists a uniquely determined directed path from the root $F_j$ to every vertex in $\mathcal{H}$. Thus, the same property applies to the alternating paths in $\mathcal{C}$, which yields that they are connected and each $F_\ell\in C_{F_j}$ occurs exactly once as final node of an alternating path. Hence, $\mathcal{C}$ fulfills all properties of a connection for~$F_j$ with respect to $\mathcal{M}$.
\end{proof}

\begin{corollary}
\label{cor:cycle-free}
All connections for $F_j$ with respect to $\mathcal{M}$ are cycle-free.
\end{corollary}
\begin{proof}
The proof follows directly from \Cref{thm:con}, using the equivalence of a connection for $F_j$ with respect to~$\mathcal{M}$ to an arborescence in the connection graph for $F_j$ with respect to~$\mathcal{M}$. Arborescences are by \Cref{def:arb} cycle-free.
\end{proof}

In the literature (e.g., \cite{gabow78}), an arborescence is also referred to as a \textit{spanning tree} of a directed graph. The sort of problem where all possible solutions to a computational problem have to be computed and explicitly returned as an output is called \textit{enumeration problem}. Methods for solving these problems are called \textit{enumeration algorithms}.

With \Cref{thm:con}, a reformulation of the initial problem (finding all connections) has been derived by showing that it is equivalent to the problem of enumerating all arborescences/spanning trees in the corresponding connection graph. Each spanning tree can then be interpreted as a connection. There are efficient algorithms to solve this enumeration problem, one of which is discussed in the next section.

\section{Enumeration of spanning trees}
\label{sec:4}

One enumeration algorithm for finding all spanning trees of a directed graph was published in \cite{gabow78}. It turns out to be an effective method for the purpose of this work and is therefore implemented in the (overall) Pantelides algorithm to solve the subproblem of finding all connections. In this section, a short summary is given of how this algorithm works. For further details of the implementation as well as theoretical results and their proofs, see \cite{gabow78}.

First, the important concept of so-called bridges has to be introduced. 
\begin{definitionn}
\cite[p.280]{gabow78}
Let $G=(V,E)$ be a directed graph and $r\in V$ a vertex.
\begin{enumerate}
    \item $G$ is called \textit{rooted at $r$} if there exists a spanning tree with root $r$ in $G$.
    \item An edge $e\in E$ is called a \textit{bridge for $r$} if $G$ is rooted at $r$ but $G\setminus\{e\}$ is not rooted at $r$.
    \item Equivalently, an edge $e\in E$ is a \textit{bridge for $r$} if it is part of every spanning tree rooted at $r$ in $G$.
\end{enumerate}
\end{definitionn}

Assume a directed graph $G=(V,E)$ and a root vertex $r$ are given and all spanning trees of $G$ rooted at $r$ have to be computed. This goal is accomplished by finding all spanning trees containing different subtrees $T\subseteq G$, also rooted at $r$.

Given a subtree $T$, the approach consists of successively adding edges to $T$ in the following way: A new edge $e_i:=(u,v)\in E$, directed from a vertex $u\in T$ to a vertex $v\notin T$, is added to $T$, and all spanning trees containing $T \cup \{e_i\}$ are computed. When this is done, the edge~$e_i$ is deleted from $G$ and $T$ and another edge~$e_j\in E\setminus\{e_i\}$ (directed from $T$ to a vertex not in $T$), is added to $T$. Again, all spanning trees containing $T \cup \{e_j\}$ are computed, then $e_j$ is deleted from $G$ and $T$. The same process continues with the next edge and is repeated until an edge is processed that is a bridge for $r$ in the modified graph $G\setminus\{e_i,e_j,...\}$.
Each spanning tree containing~$T$ has now been found exactly once.

A key point in this approach is to discover efficiently if an edge $e$ is a bridge. Assume all spanning trees containing $T\cup \{e\}$ have been computed and let $L$ be the last found spanning tree. It has to be checked if $e:=(u,v)$ is a bridge. 

There are several possibilities. The idea that is pursued in this algorithm is to consider the descendants and nondescendants of $v$ in $L$. \textit{Descendants of $v$ in $L$} are vertices that can be reached following a directed path starting in $v$ and using only edges of the spanning tree~$L$. Contrary, \textit{nondescendants of $v$ in $L$} are vertices for which there cannot be constructed such a path using edges from $L$. 

Clearly, if there is an edge in $G\setminus \{e\}$ that goes from a nondescendant of $v$ (in $L$) to $v$, then $e$ cannot be a bridge, since one could delete $e$ and replace it with that edge to construct another spanning tree. Thus, $G\setminus \{e\}$ is still rooted at $r$ and $e$ could not have been a bridge. On the other hand, if no edge that goes from a nondescendant of $v$ in $L$ to $v$ can be found, $e$ must be a bridge, because deleting $e$ leads to a graph $G\setminus\{e\}$ where there does not exist a path to vertex $v$ anymore. 

For this to hold true, the way edges are added plays an important role. Here, the algorithm adds edges depth-first. The depth of a vertex contained in a tree is the length of the path between the vertex and the root of the tree it is contained in. Thus, adding an edge depth-first means that it is added to the vertex that has the greatest depth in $T\cup\{e\}$. Particularly, this ensures that the last computed spanning tree that contains $T\cup\{e\}$ (namely the tree $L$) has the fewest descendants of $v$ amongst all spanning trees containing~${T\cup\{e\}}$. This fact can be used to prove that this bridge test works correctly (see \cite[Lemma 2, p.284]{gabow78} for more details).

Thus, it is important for the implementation to grow~$T$ depth-first. To do so, a stack $F$ is used, where edges are stored that are directed from vertices in $T$ to vertices not in $T$. Note that the action of removing an element from the top of a stack is referred to as \textit{popping}, whereas the action of adding an element to the top of a stack is referred to as \textit{pushing}. An edge $e:=(u,v)$ is always popped from the top of $F$ if it is added to $T$ and then, edges for $T\cup\{e\}$ are pushed onto the top of $F$. Also, some edges might be removed from the inner part of~$F$ while growing $T$. This is necessary for all edges in $F$ that are directed to $v$, the newest leaf of $T$. To ensure the depth-first property, these edges have to be restored at the exact same place in $F$ after all spanning trees containing~${T\cup\{e\}}$ have been found.

A second stack $FF$ is used to store already processed edges since they are temporarily deleted from $G$ but have to be restored later.

The full algorithm is stated in \Cref{alg:grow}. Note that the pseudo code uses MATLAB notation for indexing, i.e. array indexing begins at 1 and the index "end" of an array points to the last element of it or to the top of a stack. The symbol "$\triangleright$" indicates a comment in the code. \Cref{alg:enum} illustrates how to initialize the method.

\begin{algorithm}
\caption{GROW}
\label{alg:grow}
\textbf{Input:} directed graph $G=(V,E)$, directed subgraph ${T=(V_T,E_T) \subseteq G}$, stack of edges $F$, set of spanning trees~$S$\\
\textbf{Output:} set of spanning trees $S$, last computed spanning tree $L$\\
\begin{algorithmic}[1]
\IF{$|V_T| = |V|$}
\STATE $L$ $\leftarrow$ $T$ \hfill $\triangleright$ store spanning tree in $L$ and $S$
\STATE $S$ $\leftarrow$ $S \;\dot\cup\; T$
\ELSE
\STATE $FF$ $\leftarrow$ $[\emptyset]$
\WHILE{$b=0$}
\STATE $e$ $\leftarrow$ $F$(end) \hfill $\triangleright$ pop an edge $e$ from $F$, add it to $T$
\STATE $v$ $\leftarrow$ $e(2)$
\STATE $F$ $\leftarrow$ pop($\{e\}$)
\STATE $T$ $\leftarrow$ $(V_T\;\dot\cup\; \{v\},E_T\;\dot\cup\; \{e\})$
\STATE $F$ $\leftarrow$ $F\setminus\{(u,w)\in E\;|\;u\in T,w=v\}$ \hfill $\triangleright$ update $F$
\STATE $F$ $\leftarrow$ push($\{(u,w)\in E\;|\;u=v,w\notin T\}$)
\STATE $(S,L)$ $\leftarrow$ GROW$(G,T,F,S)$ \hfill $\triangleright$ recurse
\STATE $F$ $\leftarrow$ pop($\{(u,w)\in E\;|\;u=v,w\notin T\}$) \hfill $\triangleright$ restore $F$
\STATE $F$ $\leftarrow$ $F\;\dot\cup\;\{(u,w)\in E\;|\;u\in T,w=v\}$ $\triangleright$ restore in same place as before
\STATE $T$ $\leftarrow$ $(V_T\setminus \{v\},E_T\setminus \{e\})$ $\triangleright$ delete $e$ from $T$ and $G$, add it to $FF$
\STATE $G$ $\leftarrow$ $(V,E\setminus \{e\})$
\STATE $FF$ $\leftarrow$ push($\{e\}$)
\IF{$\{(u,w)\in E\;|\; w=v,\; u \textnormal{ is a nondescendant}$ 
$\textnormal{of }v\textnormal{ in }L\}\neq\emptyset$}
\STATE $b$ $\leftarrow$ 0 \hfill $\triangleright$ bridge test
\ELSE
\STATE $b$ $\leftarrow$ 1
\ENDIF
\ENDWHILE
\WHILE{$FF($end$)\neq\emptyset$}
\STATE $e$ $\leftarrow$ $FF$(end) \hfill $\triangleright$ reconstruct $G$
\STATE $F$ $\leftarrow$ push($\{e\}$)
\STATE $FF$ $\leftarrow$ pop($\{e\}$)
\STATE $G$ $\leftarrow$ $(V,E\;\dot\cup\; \{e\})$
\ENDWHILE
\ENDIF
\end{algorithmic}
\end{algorithm}

\begin{algorithm}
\caption{Enumeration of spanning trees of a directed graph}
\label{alg:enum}
\textbf{Input:} directed graph $G=(V,E)$, root node $r \in V$\\
\textbf{Output:} set of all spanning trees $S$\\
\begin{algorithmic}[1]
\STATE $T$ $\leftarrow$ $(\{r\},\emptyset)$
\STATE $F$ $\leftarrow$ push($\{(u,v)\in E\;|\; u=r\}$)
\STATE $S$ $\leftarrow$ $\emptyset$
\STATE $S$ $\leftarrow$ GROW$(G,T,F,S)$
\end{algorithmic}
\end{algorithm}

To conclude this section, the complexity of the algorithm is stated. For a directed graph~${G=(V,E)}$ that has~$N$ spanning trees, it has a time complexity of $O(|E|N)$ time and a space or memory complexity of $O(|E|)$ (see \cite[Lemma 4, p.285]{gabow78}). Next, it is shown how to use this method in the Pantelides algorithm.

\section{An algorithm that finds all connections}
\label{sec:5}

The enumeration algorithm from the previous section has to be applied to the initial problem of finding all connections. This would replace the computation stated in line 1 of Algorithm 3 from \cite[p.20]{ahrens20}. Thus, transferring the notation, the shifting graph~${G^S=(V^S_E\dot\cup V^S_V,E^S)}$, the exposed equation $F_j\in V^S_E$ and the matching $\mathcal{M}$ are given. It has been shown that finding all connections for $F_j$ with respect to~$\mathcal{M}$ is equivalent to enumerating the spanning trees with root $F_j$ in the connection graph for~$F_j$ with respect to~$\mathcal{M}$.

Therefore, the connection graph $H$ is constructed according to \Cref{def:congraph} and used as input to \Cref{alg:enum}, together with $F_j$ as the root $r$. All spanning trees in $H$ are returned. Given a spanning tree, one can reconstruct the corresponding connection by taking its edges $(F_i,F_\ell)$ and inserting into each directed edge the variable node that was assigned to the equation node $F_\ell$ by $\mathcal{M}$. That yields a set of alternating paths $(F_i,v_k,F_\ell)$ as desired. The method is summarized in \Cref{alg:con}.

\begin{algorithm}
\caption{Find all connections for $F_j$ with respect to~$\mathcal{M}$}
\label{alg:con}
\textbf{Input:} shifting graph $G^S=(V^S_E\dot\cup V^S_V,E^S)$, exposed node $F_j\in V^S_E$, matching $\mathcal{M}$ stored in \code{assign}, \code{colorE}, \code{colorV}\\
\textbf{Output:} set of all connections $P$\\
\begin{algorithmic}[1]
\STATE $C_{F_j}\leftarrow \{ F_k\in V^S_E \;|\; \exists\text{ alternating path between $F_j$}$ 
$\text{and $F_k$ in $G^S$} \}$
\STATE $V_H \leftarrow C_{F_j}\dot\cup \{F_j\}$
\STATE ${E_H \leftarrow \{(F_i,F_\ell)\in V_H \times V_H \;|\; (F_i,v_k,F_\ell) \text{ is an alterna-}}$ ${\text{ting path with } (v_k,F_\ell)\in\mathcal{M},(F_i,v_k)\in E^S\setminus \mathcal{M} \}}$
\STATE $H\;\leftarrow (V_H, E_H) $ \hfill $\triangleright$ construct connection graph
\STATE $P$ $\leftarrow$ Algorithm2($H,F_j$) \hfill $\triangleright$ enumeration algorithm
\FORALL{$T=(V_T,E_T)\in P$}
\STATE $T\leftarrow E_T$ \hfill $\triangleright$ replace spanning trees by connections
\FORALL{$e=(F_i,F_\ell)\in T$}
\STATE $e\leftarrow (F_i,v_k,F_\ell)$ such that $(v_k,F_\ell)\in\mathcal{M}$
\ENDFOR
\ENDFOR
\end{algorithmic}
\end{algorithm}

\noindent To illustrate the new method, a simple and a slightly more complex example are given in the following.

\begin{example}
Consider again the DDAE \cref{eq:example1} from \Cref{ex:con} with the shifting graph \Cref{fig:example1.1}. In \Cref{ex:congraph}, it has been shown how to construct the connection graph for $F_3$ with respect to the matching
\begin{equation*}
    \mathcal{M}=\left\{\{F_1,\{x_1,\dot x_1\}\},\{F_2,\{x_2\}\}\right\}.
\end{equation*}
It is given as $H=(V_H,E_H)$ with
\begin{align*}
    V_H&=\{F_1,F_2,F_3\},\\
    E_H&=\left\{(F_2,F_1),(F_3,F_1),(F_3,F_2)\right\},
\end{align*}
and is visualized in \Cref{fig:example2}.

Hence, one defines $G:=H$ and $r:=F_3$ as the input to \Cref{alg:enum} and enumerates all spanning trees of $G$ rooted at $r$. To initialize the process, set
\begin{align*}
    T &= (\{F_3\},\emptyset),\\
    F &= \left[(F_3,F_2),(F_3,F_1)\right],
\end{align*}
and execute \Cref{alg:grow}.

The recursion process of \Cref{alg:grow} can be visualized by the tree structure in \Cref{fig:example3}. Note that the nodes of the computation tree will be called \textit{bisections} and the edges \textit{arrows} to not confuse them with the nodes and edges of $T$ or $G$. In general, the notion of the tree is as follows: each bisection represents the current subgraph~${T\subseteq G}$, indicated by its edges $E_T$. As described in \cref{sec:4}, one then adds an edge $e\in G$ from the stack $F$ to~$T$ and computes all spanning trees containing $T\cup\{e\}$. Adding an edge $e$ is represented by an arrow pointing away from a bisection, i.e., if
\begin{equation*}
    E_T=\left\{(F_{i_1},F_{\ell_1}),...,(F_{i_m},F_{\ell_m})\right\}
\end{equation*}
and $e=(F_{i_{m+1}},F_{\ell_{m+1}})$ is added, then this is visualized in the computation tree by an arrow pointing from 
\begin{align*}
    &\left\{(F_{i_1},F_{\ell_1}),...,(F_{i_m},F_{\ell_m})\right\}\quad\text{to} \\
    &\left\{(F_{i_1},F_{\ell_1}),...,(F_{i_m},F_{\ell_m}),(F_{i_{m+1}},F_{\ell_{m+1}})\right\}.
\end{align*}
Thus, the computation of all spanning trees containing~$T$, or $T\cup\{e\}$, is represented by the subtree (of the computation tree) rooted at the bisection representing~$T$, or $T\cup\{e\}$, respectively. The arrows pointing away from~$T$ are, from left to right, all edges from the stack~$F$ that are added to $T$. Also, remember that after the computation of all spanning trees containing $T\cup\{e\}$, it has to be checked if $e$ is a bridge. If it is not, $e$ is deleted from~$T$ and $G$. This is depicted by the red arrows pointing to the bisection that represents the addition of the next edge, together with the corresponding label indicating which edge is deleted. If it is a bridge, then all spanning trees containing $T$ have been found and that iteration comes to an end. This is similarly depicted by a red arrow pointing to "END". Finally, each leaf in the lowest level of the computation tree is a complete and unique spanning tree.

\begin{figure}[t!]
    \tikzset{edge from parent/.append style={->,>=stealth}}
    \centering
    \begin{tikzpicture}[every node/.style = {align=center}]
        \tikzstyle{level 1}=[sibling distance=30mm, level distance=17mm]
        \tikzstyle{level 2}=[sibling distance=15mm, level distance=30mm]
        \Tree   [.$(\;\emptyset\;)$
                    [.\node(1){$(F_3,F_1)$};
                        [.\node(4){$(F_3,F_1)$\\$(F_3,F_2)$}; ] ]
                    [.\node(2){$(F_3,F_2)$};
                        \node(5){$(F_3,F_2)$\\$(F_2,F_1)$}; ] ]
        \draw[->,>=stealth,red] (1) to [out=315,in=225,looseness=1] node[pos=.5,below=0mm]{\small{delete}\\\footnotesize{$(F_3,F_1)$}}(2) ;
        \node(3)[right=of 2]{END};
        \draw[->,>=stealth,red] (2) to [out=315,in=225,looseness=1] node[pos=.5,below=0mm]{\footnotesize{$(F_3,F_2)$}\\\small{bridge}}(3) ;
        \node(6)[right=of 4]{END};
        \draw[->,>=stealth,red] (4) to [out=315,in=225,looseness=1] node[pos=.5,below=0mm]{\footnotesize{$(F_3,F_2)$}\\\small{bridge}}(6) ;
        \node(7)[right=of 5]{END};
        \draw[->,>=stealth,red] (5) to [out=315,in=225,looseness=1] node[pos=.5,below=0mm]{\footnotesize{$(F_2,F_1)$}\\\small{bridge}}(7) ;
    \end{tikzpicture}
\label{fig:example3.3}
\caption{Computation tree of \Cref{alg:grow} for the construction of connections for $F_3$ with respect to $\mathcal{M}$ in the shifting step of the DDAE \cref{eq:example1}.}
\label{fig:example3}
\end{figure}

In the case of the present example and as stated above, one starts with $T$ containing no edge ($E_T=\emptyset$) and pops the last element from $F$ to add it to $T$, yielding
\begin{equation}\label{it1}
\begin{aligned}
    e&=(F_3,F_1),\\
    T&=\left(\{F_3,F_1\},\{(F_3,F_1)\}\right),\text{ and}\\
    F&= \left[(F_3,F_2)\right].
\end{aligned}
\end{equation}
As all spanning trees containing $T$ shall be computed, one pops the next edge from $F$, here $(F_3,F_2)$. This results in
\begin{align*}
    e&=(F_3,F_2),\\
    T&=\left(\{F_3,F_1,F_2\},\{(F_3,F_1),(F_3,F_2)\}\right),\text{ and}\\
    F&= [\emptyset].
\end{align*}
The tree $T$ is now a complete spanning tree as it has~$n-1$ (here, $n=3$) edges. Thus, one sets $L=T$ and tests if~${e=(F_3,F_2)}$ is a bridge. The nondescendants of~$F_2$ in $L$ are $F_3$ and $F_1$, and there is no edge in $G$ that goes from a nondescendant of $F_2$ to $F_2$ besides $e$ itself. Consequently, $e$ is categorized as a bridge and indeed, all spanning trees containing the subtree with $E_T=\{(F_3,F_1)\}$ have been computed. The iteration ends and the algorithms returns to the setting of (\ref{it1}). Doing the bridge test here reveals that $e=(F_3,F_1)$ is not a bridge, since~$L$ remains unchanged and there exists the edge $(F_2,F_1)$ in~$G$ where $F_2$ is a nondescendant of $F_1$ in $L$. Therefore, the edge $(F_3,F_1)$ is deleted from $G$ and $T$ and the next iteration begins, meaning that the next edge from $F$ is added to $T$:
\begin{equation}\label{it2}
    \begin{aligned}
    e&=(F_3,F_2),\\
    T&=\left(\{F_3,F_2\},\{(F_3,F_2)\}\right),\text{ and}\\
    F&= \left[(F_2,F_1)\right].
    \end{aligned}
\end{equation}
Again, all spanning trees containing $T$ have to be computed and the next edge is popped from $F$, resulting in
\begin{align*}
    e&=(F_2,F_1),\\
    T&=\left(\{F_3,F_2,F_1\},\{(F_3,F_2),(F_2,F_1)\}\right),\text{ and}\\
    F&= [\emptyset].
\end{align*}
The tree $T$ is a new and distinct spanning tree. One sets~$L=T$ and a test reveals that $e$ is a bridge: $F_3$ is the only nondescendant of $F_1$ in $L$ that does not belong to $e$ itself, and the edge $(F_3,F_1)$ was just deleted from $G$, so it does not exist anymore in the current graph (although it will be restored later). All spanning trees containing the subtree with $E_T=\{(F_3,F_2)\}$ have been found. The current iteration is terminated and the algorithm returns to the iteration with the setting (\ref{it2}). Here, one checks if $e=(F_3,F_2)$ is a bridge and again, it is. The edge $e$ is the only edge leading to $F_2$ in $G$. Hence, this iteration ends as well, which means that all spanning trees containing the subtree with $E_T=\emptyset$ have been computed successfully, or in other words, all spanning existing in $G$. The original graph $G$ is restored (i.e., the edge that was deleted, $(F_3,F_1)$, is added once again to $G$), the whole algorithm terminates and returns
\begin{align*}
    S=\big\{&\big(\{F_3,F_1,F_2\},\{(F_3,F_1),(F_3,F_2)\}\big),\\
    &\big(\{F_3,F_2,F_1\},\{(F_3,F_2),(F_2,F_1)\}\big)\big\}.
\end{align*}
One can easily check by hand that these two spanning trees are the only ones existing in $G$. 

Finally, the result, still in tree structure, has to be converted back to a set of connections. Inserting the variable nodes stored in the matching $\mathcal{M}$ gives
\begin{align*}
    P=\big\{&\big\{(F_3,\{x_1,\dot x_1\},F_1),(F_3,\{x_2\},F_2)\big\},\\
    &\big\{(F_3,\{x_2\},F_2),(F_2,\{x_1,\dot x_1\},F_1)\big\}\big\}.
\end{align*}
By comparing $P$ to \Cref{fig:example1.2} and \Cref{fig:example1.3}, it can be seen that the algorithm successfully determined all desired connections for $F_3$ with respect to $\mathcal{M}$.
\end{example}

\begin{example}
Consider the DDAE
\begin{equation}\label{eq:example2}
\begin{aligned}
    \dot x_1 &= x_2+x_3,\\
    \dot x_2 &= x_3+ \Delta_{-\tau} x_2,\\
    \dot x_3 &= x_2+ \Delta_{-\tau} x_3,\\
    0 &= x_1 +x_2+x_3+\Delta_{-\tau} x_4.
\end{aligned}
\end{equation}
The shifting graph, after assigning $\{x_1,\dot x_1\}$ to~$F_1$,~$\{x_2,\dot x_2\}$ to $F_2$ and $\{x_3,\dot x_3\}$ to $F_3$, is shown in \Cref{fig:example4.1}. The equation $F_4$ is exposed and cannot be matched directly to any equivalence class, but it is connected via alternating paths to all other equation nodes. Therefore, it holds that ${C_{F_4}=\{F_1,F_2,F_3\}}$ and all possible connections for~$F_4$ with respect to
\begin{equation*}
    \mathcal{M}=\left\{\{F_1,\{x_1,\dot x_1\}\},\{F_2,\{x_2,\dot x_2\}\right\}, \{F_3,\{x_3,\dot x_3\}\}\}
\end{equation*}
have to be found.
The connection graph $H=(V_H,E_H)$ for $F_4$ with respect to $\mathcal{M}$ is given in \Cref{fig:example4.2} with
\begin{align*}
    V_H=\{&F_1,F_2,F_3,F_4\},\\
    E_H=\big\{&(F_1,F_2),(F_1,F_3),(F_2,F_3),(F_3,F_2),(F_4,F_1),(F_4,F_2),\\
    &(F_4,F_3)\big\}.
\end{align*}
After defining $G:=H$ and $r:=F_4$ as the input to \Cref{alg:enum} and initializing
\begin{align*}
    T &= (\{F_4\},\emptyset), \text{ and}\\
    F &= \left[(F_4,F_3),(F_4,F_2),(F_4,F_1)\right],
\end{align*}
\Cref{alg:grow} is executed to enumerate all spanning trees of $G$ rooted at $r$.

The computation tree that represents the recursion structure of \Cref{alg:grow} can be seen in \Cref{fig:example5}. Similarly to the last example, one can follow the different paths in the tree to retrace the construction of subtrees~$T$ by addition and deletion of edges $e$.

Note that even though $F$ is initialized with all three edges outgoing from $F_4$, the algorithm already terminates after the first iteration where all spanning trees containing the subtree with $E_T=\{(F_4,F_1)\}$ are com\-put\-ed. This is due to the fact that after deleting $(F_4,F_1)$ from~$G$, there is no edge leading to $F_1$ anymore, and hence it is not possible to construct another spanning tree. The following eight spanning trees are returned:
\begin{align*}
    S=\big\{&\left(\{F_4,F_1,F_2,F_3\},\{(F_4,F_1),(F_1,F_2),(F_2,F_3)\}\right),\\
    &\left(\{F_4,F_1,F_2,F_3\},\{(F_4,F_1),(F_1,F_2),(F_1,F_3)\}\right),\\
    &\left(\{F_4,F_1,F_2,F_3\},\{(F_4,F_1),(F_1,F_2),(F_4,F_3)\}\right),\\
    &\left(\{F_4,F_1,F_3,F_2\},\{(F_4,F_1),(F_1,F_3),(F_3,F_2)\}\right),\\
    &\left(\{F_4,F_1,F_3,F_2\},\{(F_4,F_1),(F_1,F_3),(F_4,F_2)\}\right),\\
    &\left(\{F_4,F_1,F_2,F_3\},\{(F_4,F_1),(F_4,F_2),(F_2,F_3)\}\right),\\
    &\left(\{F_4,F_1,F_2,F_3\},\{(F_4,F_1),(F_4,F_2),(F_4,F_3)\}\right),\\
    &\left(\{F_4,F_1,F_3,F_2\},\{(F_4,F_1),(F_4,F_3),(F_3,F_2)\}\right)\big\}.
\end{align*}
After converting them into connections with respect to the matching $\mathcal{M}$, one finally obtains
\begin{align*}
    P=\big\{&\left\{(F_4,\{x_1,\dot x_1\},F_1),(F_1,\{x_2,\dot x_2\},F_2),(F_2,\{x_3,\dot x_3\},F_3)\right\},\\
    &\left\{(F_4,\{x_1,\dot x_1\},F_1),(F_1,\{x_2,\dot x_2\},F_2),(F_1,\{x_3,\dot x_3\},F_3)\right\},\\
    &\left\{(F_4,\{x_1,\dot x_1\},F_1),(F_1,\{x_2,\dot x_2\},F_2),(F_4,\{x_3,\dot x_3\},F_3)\right\},\\
    &\left\{(F_4,\{x_1,\dot x_1\},F_1),(F_1,\{x_3,\dot x_3\},F_3),(F_3,\{x_2,\dot x_2\},F_2)\right\},\\
    &\left\{(F_4,\{x_1,\dot x_1\},F_1),(F_1,\{x_3,\dot x_3\},F_3),(F_4,\{x_2,\dot x_2\},F_2)\right\},\\
    &\left\{(F_4,\{x_1,\dot x_1\},F_1),(F_4,\{x_2,\dot x_2\},F_2),(F_2,\{x_3,\dot x_3\},F_3)\right\},\\
    &\left\{(F_4,\{x_1,\dot x_1\},F_1),(F_4,\{x_2,\dot x_2\},F_2),(F_4,\{x_3,\dot x_3\},F_3)\right\},\\
    &\left\{(F_4,\{x_1,\dot x_1\},F_1),(F_4,\{x_3,\dot x_3\},F_3),(F_3,\{x_2,\dot x_2\},F_2)\right\}\big\}\!.
\end{align*}
Indeed, all possible connections for $F_4$ with respect to~$\mathcal{M}$ have been found.
\end{example}

\begin{figure}[t!]
    \centering
    \subfigure[The shifting graph of \cref{eq:example2}.]{
        \centering
        \begin{tikzpicture} 
    		\enode[try]{1}{$F_1$}
    		\enode[try]{2}{$F_2$}
    		\enode[try]{3}{$F_3$}
    		\enode[try]{4}{$F_4$}
    		\vnode{1}{$x_1,\dot x_1$}{}
    		\vnode{2}{$x_2,\dot x_2$}{}
    		\vnode{3}{$x_3,\dot x_3$}{}
    		\vnode[low]{4}{$\Delta_{-\tau} x_2$}{}
    		\vnode[low]{5}{$\Delta_{-\tau} x_3$}{}
    		\vnode[low]{6}{$\Delta_{-\tau} x_4$}{}
    		\edge[assign]{e1}{v1}
    		\edge[assign]{e2}{v2}
    		\edge[assign]{e3}{v3}
    		\edge{e1}{v2}
    		\edge{e1}{v3}
    		\edge{e2}{v3}
    		\edge{e3}{v2}
    		\edge{e4}{v1}
    		\edge{e4}{v2}
    		\edge{e4}{v3}
    		\edge[low]{e2}{v4}
    		\edge[low]{e3}{v5}
    		\edge[low]{e4}{v6}
        \end{tikzpicture}
        \label{fig:example4.1}
    }
    \hspace{2cm}
    \subfigure[The connection graph for $F_4$.]{
        \centering
        \begin{tikzpicture}[node distance={20mm}, roundnode/.style={circle, draw=black, thin}]
    	    \node[roundnode](4){$F_4$};
    	    \node[roundnode](1)[right=of 4]{$F_1$};
    	    \node[roundnode](2)[below=of 1]{$F_2$};
    	    \node[roundnode](3)[below=of 4]{$F_3$};
    	    \draw[->,>=stealth,thick] (4) -- (1);
    	    \draw[->,>=stealth,thick] (4) -- (2);
    	    \draw[->,>=stealth,thick] (4) -- (3);
    	    \draw[->,>=stealth,thick] (1) -- (2);
    	    \draw[->,>=stealth,thick] (1) -- (3);
    	    \draw[->,>=stealth,thick] (2) to [out=170,in=10,looseness=0.5] (3);
    	    \draw[->,>=stealth,thick] (3) to [out=350,in=190,looseness=0.5] (2);
    	\end{tikzpicture}
        \label{fig:example4.2}
    }
	\caption{Visualization of the construction of connections for $F_4$ with respect to $\mathcal{M}$ in the shifting step of the DDAE \cref{eq:example2}.}
	\label{fig:example4}
\end{figure}

\begin{sidewaysfigure}[p!]
    \centering
    \begin{tikzpicture}[every node/.style = {align=center}]
        \tikzset{edge from parent/.append style={->,>=stealth}}
        \tikzstyle{level 1}=[level distance=25mm]
        \tikzstyle{level 2}=[sibling distance=15mm, level distance=30mm]
        \tikzstyle{level 3}=[sibling distance=3mm, level distance=40mm]
        \Tree   [.$(\;\emptyset\;)$
                    [.\node(6){$(F_4,F_1)$};
                        [.\node(1){$(F_4,F_1)$\\$(F_1,F_2)$}; 
                            [.\node(8){$(F_4,F_1)$\\$(F_1,F_2)$\\$(F_2,F_3)$}; ]
                            [.\node(9){$(F_4,F_1)$\\$(F_1,F_2)$\\$(F_1,F_3)$}; ]
                            [.\node(10){$(F_4,F_1)$\\$(F_1,F_2)$\\$(F_4,F_3)$}; ] ]
                        [.\node(2){$(F_4,F_1)$\\$(F_1,F_3)$};
                            [.\node(11){$(F_4,F_1)$\\$(F_1,F_3)$\\$(F_3,F_2)$}; ]
                            [.\node(12){$(F_4,F_1)$\\$(F_1,F_3)$\\$(F_4,F_2)$}; ] ]
                        [.\node(3){$(F_4,F_1)$\\$(F_4,F_2)$}; 
                            [.\node(13){$(F_4,F_1)$\\$(F_4,F_2)$\\$(F_2,F_3)$}; ]
                            [.\node(14){$(F_4,F_1)$\\$(F_4,F_2)$\\$(F_4,F_3)$}; ] ]
                        [.\node(4){$(F_4,F_1)$\\$(F_4,F_3)$};
                            [.\node(15){$(F_4,F_1)$\\$(F_4,F_3)$\\$(F_3,F_2)$}; ] ] ] ]
        \draw[->,>=stealth,red] (1) to [out=315,in=225,looseness=1] node[pos=.5,below=0mm]{\small{delete}\\\footnotesize{$(F_1,F_2)$}}(2) ;
        \draw[->,>=stealth,red] (2) to [out=315,in=225,looseness=1] node[pos=.5,below=0mm]{\small{delete}\\\footnotesize{$(F_1,F_3)$}}(3) ;
        \draw[->,>=stealth,red] (3) to [out=315,in=225,looseness=1] node[pos=.5,below=0mm]{\small{delete}\\\footnotesize{$(F_4,F_2)$}}(4) ;
        \node(5)[below right=-6mm and 2mm of 4]{\footnotesize{END}};
        \draw[->,>=stealth,red] (4) to [out=315,in=225,looseness=1] node[pos=.5,below=0mm]{\footnotesize{$(F_4,F_3)$}\\\small{bridge}}(5) ;
        \node(7)[right=of 6]{\footnotesize{END}};
        \draw[->,>=stealth,red] (6) to [out=45,in=135,looseness=1] node[pos=.5,above=0mm]{\footnotesize{$(F_4,F_1)$}\\\small{bridge}}(7) ;
        \draw[->,>=stealth,red] (8) to [out=315,in=225,looseness=1] node[pos=.5,below=0mm]{\small{delete}\\\footnotesize{$(F_2,F_3)$}}(9) ;
        \draw[->,>=stealth,red] (9) to [out=315,in=225,looseness=1] node[pos=.5,below=0mm]{\small{delete}\\\footnotesize{$(F_1,F_3)$}}(10) ;
        \draw[->,>=stealth,red] (11) to [out=315,in=225,looseness=1] node[pos=.5,below=0mm]{\small{delete}\\\footnotesize{$(F_3,F_2)$}}(12) ;
        \draw[->,>=stealth,red] (13) to [out=315,in=225,looseness=1] node[pos=.5,below=0mm]{\small{delete}\\\footnotesize{$(F_2,F_3)$}}(14) ;
        \node(16)[below right=-6mm and 2mm of 10]{\footnotesize{END}};
        \draw[->,>=stealth,red] (10) to [out=315,in=225,looseness=1] node[pos=.5,below=0mm]{\footnotesize{$(F_4,F_3)$}\\\small{bridge}}(16) ;
        \node(17)[below right=-6mm and 2mm of 12]{\footnotesize{END}};
        \draw[->,>=stealth,red] (12) to [out=315,in=225,looseness=1] node[pos=.5,below=0mm]{\footnotesize{$(F_4,F_2)$}\\\small{bridge}}(17) ;
        \node(18)[below right=-6mm and 2mm of 14]{\footnotesize{END}};
        \draw[->,>=stealth,red] (14) to [out=315,in=225,looseness=1] node[pos=.5,below=0mm]{\footnotesize{$(F_4,F_3)$}\\\small{bridge}}(18) ;
        \node(19)[below right=-6mm and 2mm of 15]{\footnotesize{END}};
        \draw[->,>=stealth,red] (15) to [out=315,in=225,looseness=1] node[pos=.5,below=0mm]{\footnotesize{$(F_3,F_2)$}\\\small{bridge}}(19) ;
    \end{tikzpicture}
    \caption{Computation tree of \Cref{alg:grow} for the DDAE \cref{eq:example2}.}
    \label{fig:example5}
\end{sidewaysfigure}
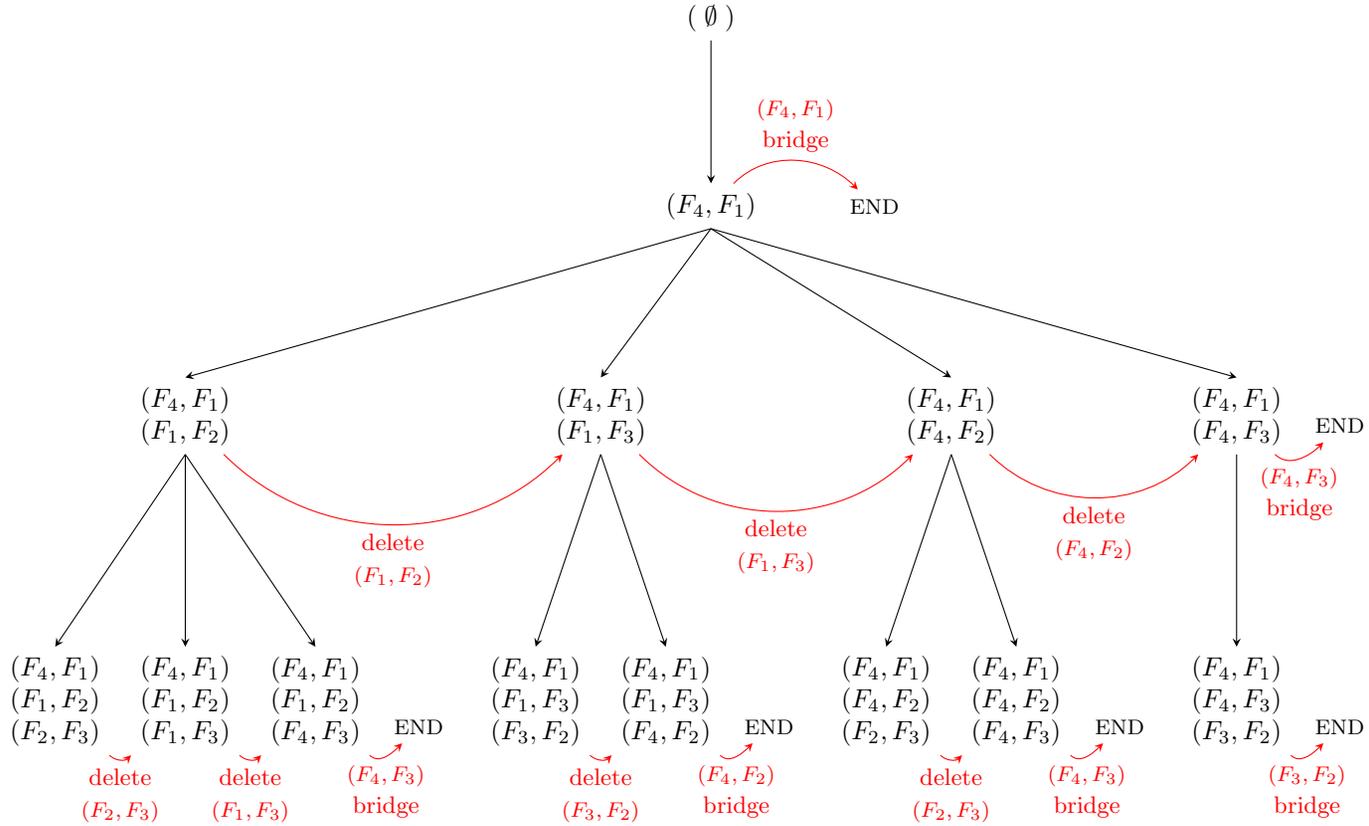

\section{Numerical demonstration}
\label{sec:6}

The developed algorithm presented in this paper has been implemented to empirically demonstrate its effectiveness. Also, a naive depth-first method is used to compute connections in order to estimate the efficiency of the new algorithm in terms of computational complexity. All computations are performed using MATLAB R2021a on a laptop with the processor Intel CORE i5-6267U CPU @2.90GHz (4 CPUs), $\sim$2.8GHz.

For simplicity, a shifting graph $G^S=(V_E^S\dot\cup V_V^S,E^S)$ is assumed to be given where only the variable nodes $v_k\in V_V^S$ of highest shift exist for $k=1,...,n-1$ (i.e., all other variable nodes have already been deleted) and each $F_i\in V_E^S$ is matched to $v_i$, for $i=1,...,n-1$. Thus, $F_n$ is exposed with respect to the matching
\begin{equation*}
    \mathcal{M}=\left\{\{F_1,v_1\},...,\{F_{n-1},v_{n-1}\}\right\}.
\end{equation*}
Three different scenarios are tested. To illustrate the edge structures $E^S$ of the corresponding shifting graphs, let $A\in\mathbb{R}^{n\times (n-1)}$ be a matrix with entries 
\begin{equation*}
    a_{ij}=
    \begin{cases}
    1,\quad\text{if } \{F_i,v_j\}\in E^S,\\
    0,\quad\text{else}.
    \end{cases}
\end{equation*}

First, a shifting graph is constructed such that 
\begin{equation}
    \label{eq:scenario1}
    A=
    \begin{bmatrix}
    1 & 1      &        &  \\
    1 & \ddots & \ddots &  \\
      & \ddots & \ddots & 1\\
      &        & 1      & 1\\
    1 & \cdots & \cdots & 1 
    \end{bmatrix}
    ,
\end{equation}
i.e., each equation node $F_i$, for $i=1,...,n-1$, is connected to at most three variable nodes and $F_n$ is connected to each $v_k$, for $k=1,...,n-1$. The computation times for shifting graphs of this structure for different~${n\in\mathbb{N}}$ can be seen in \Cref{tab:scenario1}. In all tables, "DFS" is the abbreviation for "depth-first search" and $N$ denotes the number of possible connections. Some computations have been stopped after 10 minutes of computing time, which is indicated by "$>600$". In these cases, computations for even higher $n$ have not been executed. This is marked as "-" in the tables.

\begin{table}[htb]
	\centering
	\caption{Computation times in [s] for scenario \cref{eq:scenario1}.}
	\label{tab:scenario1}
	\begin{tabular}{|c| c c c c c c|} \hline
		$n$ & 5 & 6 & 7 & 8 & 9 & 10 \\
		\hline
		DFS                & 0.01 & 0.11 & 8.3  & $>$600 & -    & -  \\
		Alg. \ref{alg:con} & 0.02 & 0.04 & 0.06 & 0.16   & 0.43 & 1.32\\ 
		$N$                & 21    & 55   & 144 & 377    & 987  & 2584\\
		\hline
	\end{tabular}
\end{table}

In a second test, a scenario is created such that 
\begin{equation}
    \label{eq:scenario2}
    A=
    \begin{bmatrix}
    1 & \cdots & 1 \\
      & \ddots & \vdots\\
      &        & 1 \\
    1 & \cdots & 1 
    \end{bmatrix}
    ,
\end{equation}
i.e., each equation node $F_i$, for $i=1,...,n-1$, is connected to $n-i$ variable nodes and $F_n$ is again connected to each $v_k$, for $k=1,...,n-1$. The computation times for shifting graphs of this structure can be seen in \Cref{tab:scenario2}.
\begin{table}[htb]
	\centering
	\caption{Computation times in [s] for scenario \cref{eq:scenario2}.}
	\label{tab:scenario2}
	\begin{tabular}{|c|c c c c c c|} \hline
		$n$ & 5 & 6 & 7 & 8 & 9 & 10 \\
		\hline
		DFS                & 0.01 & 0.22 & 31   & $>$600 & -    & -  \\
		Alg. \ref{alg:con} & 0.04 & 0.09 & 0.39 & 2.2    & 14   & 126\\ 
		$N$                & 24   & 120  & 720  & 5040   & 40320 & 362880\\
		\hline
	\end{tabular}
\end{table}

For the third scenario, a complete graph is assumed, where each equation node is connected to all variable nodes, i.e.,
\begin{equation}
    \label{eq:scenario3}
    A=
    \begin{bmatrix}
    1 & \cdots & 1 \\
    \vdots &  & \vdots\\
    1 & \cdots & 1 
    \end{bmatrix}
    .
\end{equation}
The computation times are listed in \Cref{tab:scenario3}.
\begin{table}[htb]
	\centering
	\caption{Computation times in [s] for scenario \cref{eq:scenario3}.}
	\label{tab:scenario3}
	\begin{tabular}{|c | c c c c c|} \hline
		$n$ & 5 & 6 & 7 & 8 & 9 \\
		\hline
		DFS                & 0.03 & 0.41 & 318    & $>$600 & -  \\
		Alg. \ref{alg:con} & 0.05 & 0.48 & 6.3    & 88     & 2462  \\ 
		$N$                & 125  & 1296 & 16807  & 262144 & 4782969\\
		\hline
	\end{tabular}
\end{table}

The results clearly show the advantage of \Cref{alg:con} as it is strongly superior in terms of computation time. For all scenarios, the depth-first search algorithm is only competitive for very small system sizes~$n$, before its computation time suddenly explodes. This has a simple reason: by naively testing all possible combinations of edges, an extreme amount of possibilities arises. Even more, the majority of connections computed by the depth-first search algorithm are duplicates, meaning that they possess the same alternating paths in different order. All of these have to be identified and deleted after the algorithm terminates. \Cref{alg:con}, however, does not have this problem, as only unique spanning trees (and thus, connections) are computed. Therefore, it scales well with the number of possible connections $N$ and has a huge advantage in terms of computational complexity. Nevertheless, one can also see that the problem itself is very demanding, because~$N$ increases rapidly with the system size $n$ and already for relatively small $n$, one cannot compute all connections in a reasonable time anymore. There are just too many in the case of dense graphs.

\section{Conclusion}
\label{sec:7}

In this work, the problem of finding all connections in the shifting step of the Pantelides algorithm for DDAEs from \cite{ahrens20} has been discussed. A new method, based on on the reformulation of the problem into the problem of enumerating all spanning trees (or arborescences) in a directed graph, has been developed. This directed graph is constructed with the alternating paths of the shifting graph and is called connection graph. The equivalence of the solutions to these two problems has been proven in \Cref{thm:con}. That led to the possibility to exploit the fact that there already exist efficient methods to solve the enumeration problem. By introducing and implementing the method from \cite{gabow78}, \Cref{alg:con} has been introduced to compute all connections in the shifting graph. Its effectiveness for the problem at hand has been shown by giving theoretical examples and its efficiency has been demonstrated by an implementation and numerical tests.

In summary, the lack of a satisfactory solution to the problem of finding all connections in the shifting step of the Pantelides algorithm for DDAEs has been overcome by this work for small problems. The new method now provides an efficient algorithm for its solution and will hopefully help to solve many DDAEs in the future.

\appendix
\section{Code} 
The MATLAB source code of the implementation used to compute the presented results is available as supplementary material and can be obtained under
	\begin{center}
		\href{https://github.com/DanielCollin96/pantelides\_ddae\_connections}{https://github.com/DanielCollin96/pantelides\_ddae\_connections}.
	\end{center}

\section*{Acknowledgments}
The author thanks his supervisors Ines Ahrens (Technische Universität Berlin), Benjamin Unger (Universität Stuttgart) and Volker Mehr\-mann (Technische Universität Berlin) for their help, valuable tips and the encouragement to publish this paper. His work is supported by the DFG Collaborative Research Center 910 \textit{Control of self-organizing nonlinear systems: Theoretical methods and concepts of application},
project number 163436311.

\bibliographystyle{siamplain}
\bibliography{references}
\end{document}